\documentclass[11pt, reqno]{amsart}
\usepackage{bm}
\usepackage{verbatim}

\usepackage{amsmath,amssymb,amscd,amsthm,amsxtra, esint, bbm}
\usepackage[dvips]{graphics,epsfig}

\allowdisplaybreaks[2]

\newtheorem{theorem}{Theorem} [section]
\newtheorem{maintheorem}{Theorem}
\newtheorem{lemma}[theorem]{Lemma}
\newtheorem{proposition}[theorem]{Proposition}
\newtheorem{remark}[theorem]{Remark}

\newcommand{\noi}{\noindent}

\newcommand{\R}{\mathbb{R}}
\newcommand{\C}{\mathbb{C}}
\newcommand{\T}{\mathbb{T}}

\newcommand{\al}{\alpha}
\newcommand{\dl}{\delta}

\newcommand{\eps}{\varepsilon}

\newcommand{\g}{\gamma}

\newcommand{\ld}{\lambda}
\newcommand{\Ld}{\Lambda}
\newcommand{\s}{\sigma}
\newcommand{\ft}{\hat}
\newcommand{\wt}{\widetilde}
\newcommand{\cj}{\overline}
\newcommand{\dx}{\partial_x}
\newcommand{\dt}{\partial_t}
\newcommand{\dd}{\partial}

\newcommand{\too}{\longrightarrow}

\newcommand{\jb}[1]
{\langle #1 \rangle}

\newcommand{\proj}{\mathbb{P}}

\newcommand{\ind}{\mathbf 1}

\numberwithin{equation}{section}
\numberwithin{theorem}{section}

\title[On invariant Gibbs measures conditioned on mass and momentum]
{On invariant Gibbs measures conditioned \\on mass and momentum}

\author{Tadahiro Oh, Jeremy Quastel}
\address{Tadahiro Oh\\ Department of Mathematics\\ 
Princeton University\\Fine Hall\\
Washington Rd.\\
Princeton, NJ 08544-1000, USA}
\email{hirooh@math.princeton.edu}

\address{Jeremy Quastel\\ Departments of Mathematics and Statistics\\
University of Toronto\\
40 St. George St, Toronto, ON M5S 2E4, Canada}
\email{quastel@math.toronto.edu}
\thanks{J.~Quastel was partially supported by the Natural Sciences and Engineering Research Council of Canada.}

\begin{document}

\begin{abstract} We construct a Gibbs measure
for the  nonlinear Schr\"odinger equation (NLS) on the circle,
conditioned on prescribed mass and momentum:
\[d \mu_{a,b} = Z^{-1} \ind_{ \{\int_\T |u|^2 = a\} }\ind_{\{i \int_\T u \cj{u}_x  = b\}}
e^{ \pm  \frac{1}{p}\int_\T |u|^p-\frac{1}{2}\int_{\T} |u|^2  } d P\]

\noi
for $a \in \R^+$ and $b \in \R$,
where $P$ is the complex-valued Wiener measure on the circle.
We also show that $\mu_{a,b}$ is invariant under the flow of NLS.  We note that $i \int_\T u \cj{u}_x$
is  the L\'evy stochastic area, and
 in particular that this is invariant under the flow of NLS.
\end{abstract}



\maketitle

\tableofcontents

\section{Introduction}
\label{SEC:1}

We consider 
the periodic  nonlinear Schr\"odinger equation (NLS) on the circle:
\begin{equation} \label{NLS1}
i u_t + u_{xx}  = \pm |u|^{p-2} u,  \qquad
(x, t) \in \T\times \R
\end{equation}

\noi
where $\T = \R /  \mathbb{Z}$.
Recall that \eqref{NLS1} is a Hamiltonian PDE with Hamiltonian:
\begin{equation} \label{Hamil}
H(u) = \frac{1}{2}\int_\T |u_x|^2 \pm \frac{1}{p} \int_\T |u|^{p}.
\end{equation}

\noi
Indeed, \eqref{NLS1} can be written as
\begin{equation} \label{NLS2}
u_t = i \frac{\dd H} {\dd \bar{u}}.
\end{equation}

\noi
Recall that \eqref{NLS1} also conserves the mass $M(u) = \int |u|^2$ and 
the momentum $P(u) = i \int u \cj{u}_x$.
Moreover, the cubic NLS ($p = 4$) is known to be completely integrable \cite{ZS, GKP}
in the sense
that it enjoys the Lax pair structure 
and thus there exist infinitely many conservation laws for \eqref{NLS1}.
For general $p\ne4$, the mass $M$, the momentum $P$, and  the Hamiltonian $H$
are the only known conservation laws.
Our main goal in this paper is 
to construct an invariant Gibbs measure conditioned on mass and momentum.

\medskip
First, consider a Hamiltonian flow on $\mathbb{R}^{2n}$:
\begin{equation} \label{HR2}
\dot{p}_i = \tfrac{\partial H}{\partial q_j}, \quad
\dot{q}_i = - \tfrac{\partial H}{\partial p_j} 
\end{equation}

\noi
 with Hamiltonian $ H (p, q)= H(p_1, \dots, p_n, q_1, \dots, q_n)$.
 Then, Liouville's theorem states that the Lebesgue measure 
 $\prod_{j = 1}^{n} dp_j dq_j$ on $\mathbb{R}^{2n}$ is invariant under the flow.
Then, it follows from the conservation of the Hamiltonian $H$
that  the Gibbs measure $e^{-  H(p, q)} \prod_{j = 1}^{n} dp_j dq_j$ 
is invariant under the flow of \eqref{HR2}.
Now note that if $F(p, q)$ is any (reasonable) function that is conserved under the flow of \eqref{HR2}, 
then the measure $d \mu_F = F(p, q)e^{- H(p, q)} \prod_{j = 1}^{n} dp_j dq_j$ is also invariant.

By viewing \eqref{NLS1}
as an infinite dimensional Hamiltonian system, 
one can consider the issue of invariant Gibbs measures for \eqref{NLS1}.
Lebowitz-Rose-Speer \cite{LRS} constructed Gibbs measures of the form 
\begin{equation} \label{Gibbs}
d \mu = Z^{-1} e^{- H(u)} \prod_{x\in \T} d u(x)
= Z^{-1} e^{\mp \frac{1}{p} \int_\T |u|^{p}}
\underbrace{e^{-\frac{1}{2}\int_\T |u_x|^2} \prod_{x\in \T} d u(x)}_{= \text{ Wiener measure } P}
\end{equation}

\noi
as a weighted Wiener measure on $\T$.
In the focusing case, i.e. with  the plus sign in \eqref{Gibbs}, 
the result  holds only for $p \leq 6$ 
with an $L^2$-cutoff $\ind_{\{\int |u|^2 \leq B \}}$,
where $B$ is any positive number when $p < 6$
and $B < \|Q\|^2_{L^2(\R)}$ when $p = 6$.
Here,  $Q$ is the ground state of the following elliptic equation:
\begin{equation}\label{ground}
 (p-2) Q'' - (p+2) Q + Q^{p-1} = 0.
\end{equation} 
 
\noi 
By analogy with the finite dimensional case, 
we expect such a Gibbs measure $\mu$ is invariant under the flow of \eqref{NLS1}.
(Recall that the $L^2$-norm is conserved.)
In addressing the question of invariance of $\mu$, 
we need to have a well-defined flow on the support of $\mu$.
However, as a weighted Winer measure,
the regularity of $\mu$ is inherited
from that of the Wiener measure.
i.e. $\mu$ is supported on $H^s(\T) \setminus H^{\frac{1}{2}}(\T)$, $s <\frac{1}{2}$.
In \cite{B1}, Bourgain proved local well-posedness of \eqref{NLS1}
\begin{itemize}
\item in $L^2(\T)$ for (sub-)\,cubic NLS ($p \leq 4$),
\item in $H^s(\T)$, $s > 0$, for  (sub-)\,quintic NLS ($4< p \leq 6$),
\item in $H^s(\T)$, $ s >  \frac{1}{2}-\frac{1}{p}$, for $p >6$.
\end{itemize}

\noi
Using the Fourier analytic approach, he \cite{B2} 
continued the study of Gibbs measures and 
proved the invariance of $\mu$ under the flow of NLS.

\medskip

Once the invariance of the Gibbs measure $\mu$ is established, 
we can regard the flow map of \eqref{NLS1}
as a measure-preserving transformation on an (infinite-dimensional) phase space, say $H^{\frac{1}{2}-\epsilon}$,
equipped with the Gibbs measure $\mu$. 
Then, it follows from 
Poincar\'e recurrence theorem 
that almost all the points of the phase space are stable according to Poisson \cite{Z},  
i.e. if $\mathcal{S}_t$ denotes a flow map of \eqref{NLS1}: $u_0 \mapsto u(t) = \mathcal{S}_{t} u_0$, 
then for almost all $u_0$, there exists a sequence $\{t_n\}$ tending to $\infty$
such that $ \mathcal{S}_{t_n} u_0 \to u_0$.
Moreover, such dynamics is also multiply recurrent
in view of Furstenberg's multiple recurrence theorem \cite{F}: 
let $A$ be any measurable set with $\mu(A) > 0$.
Then, for any integer $k >1$, there exists $n \ne0$
such that 
$\mu( A \cap \mathcal{S}_n A
\cap \mathcal{S}_{2n} A \cap\cdots \cap\mathcal{S}_{(k-1)n} A )>0$.
Note that this recurrence property is known to hold only in the support of the Gibbs measure,
i.e. not for smooth functions.

Then, one of the natural questions, posed by Lebowitz-Rose-Speer \cite{LRS} and Bourgain \cite{B4},
is the ergodicity of the invariant Gibbs measure $\mu$. 
i.e. is the phase space irreducible under the dynamics,
or can it be decomposed into disjoint subsets, 
where the dynamics is recurrent within each disjoint component?
In order to ask such a question, 
one needs to prescribe the $L^2$-norm
since it is an integral of motion for \eqref{NLS1}.
It is not difficult to see that 
the momentum is also finite almost surely on the support of the Gibbs measure.
Indeed, if $u$ is distributed according to the Wiener measure,
then it can be represented as\footnote{We ignore the zero-frequency issue here.
See \eqref{G2} below.} 
\begin{equation} \label{G1}
u(x) = \sum_{n\ne 0 } \frac{g_n}{2\pi n} e^{2\pi inx},
\end{equation}

\noi
where $\{g_n\}_{n\ne 0}$ is a family of independent standard complex-valued
Gaussian random variables, i.e.
its real and imaginary parts are independent Gaussian random variables
with mean zero and variance 1.
Then, we can write the momentum as
\begin{align*}
P(u) = i \int u\cj{u}_x
= \sum_{n\ne0} \frac{|g_n(\omega)|^2}{2\pi n} 
= \sum_{n\geq 1} \frac{|g_n(\omega)|^2 - |g_{-n}(\omega)|^2}{2\pi n}.
\end{align*}

\noi
Thus, we have $\mathbb{E}\big[\big(P(u)\big)^2\big] \lesssim \sum_{n\geq 1} n^{-2} <\infty.$\footnote{We
use $A \lesssim B$ to denote an estimate of the form $A \leq CB$ for some $C>0$.
Similarly, we use $A\sim B$ to denote $A \lesssim B$ and $B \lesssim A$.}
Hence, $|P(u)|<\infty$ a.s.
In the following, we construct
 invariant Gibbs measures with prescribed $L^2$-norm and momentum
  as the first step in 
  studying finer dynamical properties of the NLS flow equipped with the invariant Gibbs measure,
viewed as an infinite-dimensional dynamical system with a measure-preserving transformation.

\begin{remark}\label{REM:cubic}\rm
Recall that the cubic NLS ($p = 4$) is completely integrable. 
Hence, it makes sense to pose a question of ergodicity only for $p \neq 4$.
See \cite{LRS}.

There are infinitely many conservation laws for the cubic NLS,
with the leading term of the form $\int_\T |\dx^k u|^2 dx$, 
roughly corresponding to the $H^k$-norm, 
and of the form $\int_\T u \, \dx^{2k +1}\cj{u} \, dx$, $k \in \mathbb{N}\cup\{0\}$. 
See \cite{FT, ZM}.
By \eqref{G1}, we can easily see that 
all these conservation laws, except for the $L^2$-norm and momentum, 
are almost surely divergent under the Gibbs measure. 
Thus, it may seem that the $L^2$-norm and momentum are the only conserved quantities
which are finite a.s.~in the support of the Gibbs measure.
However, from a different perspective, 
we have a different set of infinitely many 
conserved quantities for \eqref{NLS1},
namely the spectrum of the Zakharov-Shabat operator $L$ (also called  the Dirac operator)
appearing in the Lax pair formulation of \eqref{NLS1}:
$\dt L = [B, L]$ (with some appropriate $B$.)   These are finite under the Gibbs measure.
Expressing the flow of \eqref{NLS1} 
in the Liouville coordinates (or rather in the Birkhoff coordinates)
with actions and angles (which are determined in terms of the spectral data),
the flow basically becomes trivial.
See \cite{GKP}.
\end{remark}

In constructing a Gibbs measure conditioned on mass and momentum, 
we first condition the Wiener measure on mass and momentum.
Recall that if $u$ is distributed 
according to the Wiener measure $P$ given by\footnote{The mass is added
to take care of  the zeroth frequency. We still refer to $P$ in \eqref{Gauss1}
and
$u$ in \eqref{G2}
as the Wiener measure
and  the Brownian motion, respectively.} 
\begin{equation} \label{Gauss1}
 dP = Z^{-1} e^{-\frac{1}{2}\int_\T |u|^2  -\frac{1}{2}\int_\T |u_x|^2} \prod_{x\in \T}du(x),
\end{equation}

\noi
then it can be represented as
\begin{equation} \label{G2}
u(x) = \sum_{n\in \mathbb{Z} } \frac{g_n}{\sqrt{1+4\pi^2n^2}} e^{2\pi inx},
\end{equation}

\noi
where $\{g_n\}_{n\in \mathbb{Z}}$ is a family of independent standard complex-valued
Gaussian random variables.
Note that \eqref{G2} is basically the Fourier-Wiener series
for the Brownian motion (except for the zeroth mode.)
Given $a> 0$ and $b \in \R$,  
define the conditional Wiener measures $P_\eps  = P_{\eps, a, b} $, $\eps > 0$,  as follows.
Given a measurable set $E$,  we define $P_\eps(E)$ by
\begin{equation} \label{Gauss3}
P_\eps(E) = P\bigg(  E\, \Big| 
\int_\T |u|^2 \in A_\eps(a), 
\,   i \int_\T u \cj{u}_x  \in B_\eps(b)\bigg),
\end{equation}

\noi
where $A_\eps(a)$ and $B_\eps(b)$
are neighborhoods shrinking nicely\footnote{See Subsection \ref{SUBSEC:2.1}
for the definition.} to $a$ and $b$ as $\eps \to 0$.  Here $P(C\mid D) = P(C\cap D)/P(D)$ is the
standard, naive, conditional probability given by Bayes' rule.
In terms of the density, we have 
\begin{equation} \label{Gauss2}
 dP_\eps  = \ft{Z}_\eps^{-1} \ind_{\{\int_\T |u|^2 \in A_\eps(a)\}}
 \ind_{\{ i \int_\T u \cj{u}_x \in B_\eps(b) \}} dP.
\end{equation}

\noi
Now, we would like to define the conditioned measure 
\[ P_0 (E) = P_{0, a, b}(E) = 
P\bigg(  E\, \Big| 
\int_\T |u|^2 = a, 
\,   i \int_\T u \cj{u}_x  = b\bigg)\]

\noi
by $P_0 = \lim_{\eps\to 0} P_\eps$.
Namely, we define $P_0$ by 
\begin{equation} \label{Gauss4}
 P_0(E) := \lim_{\eps\to0}P\bigg(  E\, \Big| 
\int_\T |u|^2 \in A_\eps(a), 
\,   i \int_\T u \cj{u}_x  \in B_\eps(b)\bigg).
\end{equation}

\noi
Note that the normalization constant $\ft{Z}_\eps$ in \eqref{Gauss2} tends to 0
as $\eps \to 0$.
Hence, some care is needed. 
We discuss details in Subsection 2.1.

\medskip

Finally, we define
the conditioned Gibbs measure $\mu_0 = \mu_{a, b}$
in terms of the Wiener measure $P_0 = P_{0, a, b} $ conditioned on mass and momentum, 
by setting
\begin{equation} \label{Gibbs1}
d\mu_0 = Z_0^{-1} e^{\mp \frac{1}{p} \int_\T |u|^p }dP_0.
\end{equation}

\noi
In the defocusing case, this clearly
defines a probability measure 
since $e^{-\frac{1}{p} \int_\T |u|^p } \leq 1$.
In the focusing case, 
we need to show that 
\begin{equation} \label{weight1}
e^{\frac{1}{p} \int_\T |u|^p } \in L^1(d P_0).
\end{equation}

\noi
Lebowitz-Rose-Speer \cite{LRS}
and Bourgain \cite{B2} proved a similar integrability result 
of the weight $e^{\frac{1}{p} \int_\T |u|^p }$
with respect to the (unconditioned) Wiener measure $P$ defined in \eqref{Gauss1}.
Bourgain's argument was
based on dyadic pigeonhole principle
and a large deviation estimate
(see Lemma 4.2 in \cite{OQV}.)
In Subsection \ref{SUBSEC:2.2}, we follow Bourgain's argument
and prove \eqref{weight1} by dyadic pigeonhole principle
and a large deviation estimate for $P_0$.
This large deviation estimate for $P_0$ 
is by no means automatic,
and we need to deduce it by establishing  a {\it uniform} large deviation estimate
for the conditioned Wiener measures $P_\eps$, $\eps > 0$
(see Lemma \ref{LEM:devi} below.)
As a result, we obtain the $L^1$-boundedness result
\[\mathbb{E}_{P_\eps} \Big[e^{\frac{1}{p} \int_\T |u|^p }\Big] \leq C_p <\infty\]

\noi
for all sufficiently small $\eps \geq 0$.
We point out that the proof of 
Lemma \ref{LEM:devi} (and hence the argument
in Subsection \ref{SUBSEC:2.1}) is the heart of this paper.

We state the main theorem.
The proof is presented in  in the next section.

\begin{maintheorem} \label{thm1}
Let $a>0$ and $b \in \R$.
For $p> 2$, let $\mu_0$ be the Gibbs measure $\mu_0 = \mu_{ a, b}$ conditioned on mass and momentum defined in \eqref{Gibbs1}.
Also, assume that $p\leq 6$ in the focusing case.
Then, $\mu_0$ is a well-defined probability measure
(with sufficiently small mass $a$ when $p = 6$ in the focusing case),
absolutely continuous to the conditioned Wiener measure $P_0$.
Moreover, $\mu_\eps$ converges weakly to $\mu_0$ as $\eps \to 0$,
where
$\mu_\eps$ is defined by  
\begin{equation} \label{Gibbs2}
d\mu_\eps := Z_\eps^{-1} e^{\mp \frac{1}{p} \int_\T |u|^p }dP_\eps.
\end{equation}

\end{maintheorem}

\begin{remark} \rm
In the critical case, i.e.  focusing with $p = 6$, 
Lebowitz-Rose-Speer \cite{LRS} proved that 
the weight $\ind_{\{\int_\T |u|^2 \leq B\}}e^{\frac{1}{p}\int_\T|u|^p}$ 
is integrable with respect to the (unconditioned) Wiener measure $P$ in \eqref{Gauss1}
as long as $B < \|Q\|^2_{L^2(\R)}$, where $Q$ is the ground state
for \eqref{ground}.
Indeed, this is sharp (except for the endpoint $B = \|Q\|^2_{L^2(\R)}$.)
By Fourier analytic techniques, Bourgain \cite{B2} provided another proof of this $L^1$-boundedness result.
However, his argument does not allow us to determine the (sharp) upperbound on the size $B$ of the $L^2$-cutoff in the critical case.
We believe that, in the critical case,  the upperbound on $a = \int_\T|u|^2dx$ 
in Theorem \ref{thm1}
is also given by $\|Q\|^2_{L^2(\R)}$.
Unfortunately, our proof of Theorem \ref{thm1}, following Bourgain's idea,
does not provides such a quantitative bound.
\end{remark}

It follows from invariance of the Gibbs measure $\mu$ in \eqref{Gibbs}
(with an $L^2$-cutoff in the focusing case)
and the conservation of mass and momentum
that $\mu_\eps$ is invariant under the flow of \eqref{NLS1}
for each {\it fixed} $\eps > 0$. 
As a corollary to Theorem \ref{thm1}, 
we obtain invariance of the conditioned Gibbs measure $\mu_0$.

\begin{maintheorem} \label{thm2}
Let $a>0$, $b \in\R$,
and $p >2$ be as in Theorem \ref{thm1}.
Then, the conditioned Gibbs measure $\mu_0 = \mu_{a, b}$  defined in \eqref{Gibbs1} 
is invariant under the flow of NLS \eqref{NLS1}.
\end{maintheorem}

We conclude this introduction with several remarks.
The first is  about conditional probabilities. 

\begin{remark}\rm A natural way
 to proceed with this construction 
 is to start with the (unconditioned) Gibbs measure $\mu$ in \eqref{Gibbs} 
 on the space $\Omega$,  which is the space of continuous complex-valued  functions on the circle,
with the topology of uniform convergence and the Borel $\sigma$-field $\mathcal{F}$.  This is a 
complete separable metric space.  
Let $\mathcal{G}$ be the sub $\sigma$-field generated by the measurable maps $\int_\T |u|^2$ and
$ i \int_\T u \cj{u}_x$.  
There is a general theorem 
which guarantees the existence of a  conditional probability, i.e. a family of measures $\mu_{u}$, $u\in \Omega$ such that 
(i)  for any $A\in \mathcal{F}$,  $\mu_{u}(A)$ is measurable with respect to $\mathcal{G}$ as a function of $u$; (ii)  for any $A\in \mathcal{G}$ and $B\in \mathcal{F}$, $\mu(A\cap B)= E_\mu[ {\bf 1}_A \mu_u(B)]$.   
It follows from (i) and (ii) that given $B\in \mathcal{F}$, we have
\begin{equation} \label{COND1}
\mu_{u}(B)=
\mu_{\int_\T |u|^2,\, i \int_\T u \cj{u}_x}(B) 
\end{equation}

\noi
for $\mu$-almost every $u$. 
The sets of measure zero, on  which \eqref{COND1} fails, 
depend on $B \in \mathcal{F}$,
and thus their union could be a set of nontrivial measure.  
Hence, one needs some regularity.  The best that can be said in such a general context is that if  $\mathcal{G}$ is countably generated (and one can check that ours is),  then $\mu_{u}$ is a {\it regular} conditional probability in the sense that (iii) $\mu_u(A)={\bf 1}_A(u)$ for $A\in \mathcal{G}$.  In our context, 
this reassures us that our conditioned Gibbs measure  $\mu_0 = \mu_{a,b}$ gives mass one to $u$ with $\int_\T |u|^2=a$ and $i \int_\T u \cj{u}_x=b$.  
However, we only know that this property holds for {\it almost every} $a$ and $b$, 
and there is no soft way out to obtain the same for {\it all} $a$ and $b$.  (Another way to think
of this is that applying the Lebesgue differentiation theorem to (ii) gives Theorem \ref{thm1} for almost every $a$ and $b$.)  
Since we want our conditioned measures to be defined for every value
of $a$ and $b$, we have to define them directly.  
For the conditioned Wiener measure $P_0$, which is just a Gaussian measure, this is straightforward.  
In this case, we can even use the fact that the distributions of $a$ and $b$ are basically explicit.  
However, for the Gibbs measure $\mu_{a, b}$,  it requires hard analysis.

\end{remark}

\begin{remark}\rm
Consider the (generalized) Korteweg-de~Vries equation (gKdV):
\begin{equation} \label{KDV}
u_t + u_{xxx} = \pm u^{p-2}  u_x .
\end{equation}

\noi
For an integer $p \geq 3$, \eqref{NLS1} is a Hamiltonian PDE with Hamiltonian:
\begin{equation} \label{Hamil2}
H(u) = \frac{1}{2}\int_\T u_x^2 \pm \frac{1}{p} \int_\T u^{p},
\end{equation}

\noi
and \eqref{KDV} can be written as
$u_t = \dx \frac{d H} {d u}$.
Also recall that \eqref{KDV} preserves the mean $\int_\T u$
and the $L^2$-norm.
Bourgain \cite{B2} constructed Gibbs measures of the form \eqref{Gibbs}
(with an appropriate $L^2$-cutoff
$\ind_{\{\int |u|^2 \leq B \}}$
unless it is defocusing when $p$ is even)
for \eqref{KDV},
and proved its invariance under the flow
for $p = 3,  4$.
Recently, Richards \cite{R} established invariance of the Gibbs measure
for \eqref{KDV} when $p = 5$.
In an attempt to study more dynamical properties of \eqref{KDV}, 
one can construct Gibbs measure conditioned on mass
by an argument similar to Theorem \ref{thm1}.
In this case, an analogue of Theorem \ref{thm1} holds 
for all (even) $p$ when \eqref{KDV} is defocusing,
and for $p \leq 6$
when it is non-defocusing.
However, an analogue of Theorem \ref{thm2}
holds only for $p\leq 5$ due to lack of well-defined flow for gKdV \eqref{KDV} in the support
of the Gibbs measure when $p\geq 6$.
%
Note that KdV ($p = 3$)
and mKdV ($p = 4$)
are  completely integrable.
Hence, a question of ergodicity can be posed only for $p \geq 5$.
See Remark \ref{REM:cubic}.
\end{remark}

\begin{remark}\rm
An interesting but straightforward comment is that the momentum $P(u)$
is nothing but the L\'evy stochastic area 
of the planar  loop $(\text{Re}\,u(x), \text{Im}\,u(x))$, $0\le x<2\pi$,
\begin{align}\label{five}
P(u) & = i \int_{\T} u \cj{u}_x
= \int_{\T} (\text{Re}\,u) \, d (\text{Im}\,u) - (\text{Im}\,u) \, d (\text{Re}\,u).
\end{align}
Note that this is not the actual area enclosed by the loop, but a signed version.  
A Brownian loop has infinitely many self-intersections.  
Regularizing the Brownian loop gives a loop with finitely many self-intersections.
The `area' is then computed through the path integral above, with each subregion bounded by non-intersecting part of the loop having area counted positive or negative depending on  whether the boundary is traversed in 
the counterclockwise or clockwise direction, respectively.  This includes the fact that the areas inside
internal loops are multiply counted.  Removing the regularization gives the 
L\'evy stochastic area.  Remarkably, unlike other stochastic integrals, the limit does not depend on
the regularization procedure.  For example, one can check directly that the It\^o (left endpoint rule in the Riemann sum) and Stratonovich (midpoint rule) versions
of  \eqref{five} give the same result.  The stochastic area has attracted a great deal of attention.
L\'evy \cite{L} found the exact expression $\tfrac14 (\cosh (x/2))^{-2}$ for its density under the standard Brownian motion measure.  Our base Gaussian measure \eqref{Gauss1} is almost the same as the standard Brownian
motion, and the analogous computation can be performed (see Section 2.1.)  Our Gibbs measures
$\mu_0 = \mu_{a,b}$  are absolutely continuous with respect to the base Brownian motion, so most of the results about the stochastic area continue to hold, though, of course, there are no longer any exact formulas.  
The L\'evy area is basically the only new element  when one moves from 
the Wiener-It\^o chaos of order one to order two.
Therefore, it is a natural object to supplement the Brownian path itself, and this is the basis of the rough path theory \cite{LQ}.  It seems
a remarkable fact that the flow of  NLS preserves the 
 L\'evy area.

\end{remark}

\noindent
{\bf Acknowledgments:}
The authors would like to thank the anonymous referee for 
pointing out an error in the previous version of this paper
as well as for helpful comments.

\section{Proof of Theorem \ref{thm1}:
Construction of the conditioned Gibbs measures}
\subsection{Wiener measure conditioned on mass and momentum}
\label{SUBSEC:2.1}

In this subsection, we
construct the Wiener measure $P_0$ conditioned on mass $a$ and momentum $b$
for any {\it fixed} $a >0$ and $b \in \R$.
Given $P_\eps$ as in \eqref{Gauss2},
we define $P_0$ as a limit of $P_\eps$ by \eqref{Gauss4},
where $E$ is an arbitrary set in the $\s$-field $\mathcal{F}$.
In the following, we show that \eqref{Gauss4} indeed defines a probability measure.
For this purpose, we can simply take $E$ to be in some generating family of $\mathcal{F}$.
Let us choose
the increasing family $\mathcal{F}_N = \s ( g_n ; |n| \leq N)$
as such a generating family of $\mathcal{F}$.

Fix a nonnegative integer $N$ and a Borel set $F$ in $\mathbb{C}^{2N+1}$.
Let $E = \{ \omega: (  g_n; |n|\leq N ) \in F\}$.
Then, by \eqref{Gauss3}, we have
\begin{align*} 
P_\eps(E) & = P\bigg(  (g_n; |n| \leq N) \in F \, \Big| 
\int_\T |u|^2 \in A_\eps(a), 
\,   i \int_\T u \cj{u}_x  \in B_\eps(b)\bigg),
\end{align*}

\noi
where $A_\eps(a)$ and $B_\eps(b)$
are neighborhoods shrinking nicely to $a$ and $b$ as $\eps \to 0$.
That is, 
\begin{itemize}
\item[(a)] For each $\eps > 0$, we have
 \[A_\eps(a) \subset (a-\eps, a+ \eps) 
\quad \text{and} \quad  B_\eps(b) \subset (b-\eps, b+\eps).\]

\noi

\item[(b)] There exists $\al>0$, independent of $\eps$, 
such that 
\[|A_\eps(a)| > \al \eps
\quad \text{and} \quad |B_\eps(b)| > \al \eps.\]
\end{itemize}

\noi
By \eqref{G2}, we have
\begin{equation} \label{Xplancherel}
\int_\T |u(x)|^2 dx = \sum_{n\in\mathbb{Z}}  \jb{\wt{n}}^{-2} |g_n|^2
\quad \text{and}
\quad 
i \int_\T u \cj{u}_x dx = \sum_{n\in \mathbb{Z}} \jb{ \wt{n}}^{-2}\wt{n} |g_n|^2 ,
\end{equation}

\noi
where $\wt{n} = 2\pi n$ and  $\jb{\wt{n}} = \sqrt{1+\wt{n}^2}$. 
Therefore, by independence of $\{g_n\}_{|n|\leq N}$ and $\{g_n\}_{|n|\geq N+1 }$, we have
\begin{align}
P_\eps(E) & = \int_F 
\frac{P\Big( 
\sum_{|n|\geq N+1} \jb{\wt{n}}^{-2} |g_n|^2\in A_\eps(\wt{a}), 
\,   \sum_{|n|\geq N+1} \jb{ \wt{n}}^{-2}\wt{n} |g_n|^2  \in B_\eps(\wt{b})\Big)}
{P\Big( \sum_{n\in \mathbb{Z}} \jb{\wt{n}}^{-2} |g_n|^2\in A_\eps(a), 
\,   \sum_{n\in \mathbb{Z}} \jb{\wt{n}}^{-2} \wt{n} |g_n|^2  \in B_\eps(b)\Big)}
\label{Wiener1}
\\
& \hphantom{XXXXXXX}
\times \frac{e^{-\frac{1}{2}\sum_{|n|\leq N}  |\xi_n|^2}}{(2\pi)^{2N+1}}\prod_{|n|\leq N} d \xi_n,
\notag 
\end{align}

\noi
where $d \xi_n$ denotes the Lebesgue measure on $\mathbb{C}$,
and $A_\eps(\wt{a})$ and $B_\eps(\wt{b})$
are the translates of $A_\eps(a)$ and $B_\eps(b)$
centered at 
\begin{equation}
\wt{a} = a - \sum_{|n| \leq N} \jb{\wt{n}}^{-2} |\xi_n|^2, \quad
\text{and} \quad
\wt{b} = b - \sum_{|n| \leq N} \jb{\wt{n}}^{-2} \wt{n} |\xi_n|^2,
\end{equation}

\noi
respectively.

Now, define the density $f_N(a, b)$
by 
\begin{equation}\label{Xdensity}f_N(a, b) \, da  db = P\bigg( 
\sum_{|n|\geq N} \jb{\wt{n}}^{-2} |g_n|^2\in da, 
\,   \sum_{|n|\geq N} \jb{\wt{n}}^{-2} \wt{n} |g_n|^2  \in db\bigg).
\end{equation}

\noi
Then, we have the following lemma on the regularity of $f_N$. 
\begin{lemma} \label{LEM:Xft}
Let $\ft{f}_N$ be the characteristic function (Fourier transform) of $f_N$.
Then, we have $\ft{f}_N \in L^1(\R^2)$
with estimate: $\|\ft{f}_N\|_{L^1(\R^2)} <C(N) <\infty$,
where $C(N)$ is at most a power of $N$.
In particular, $f_N$ is bounded and uniformly continuous.
\end{lemma}

\begin{proof}
By computing the characteristic function of $f_N$, we have
\begin{align}
\ft{f}_N(s, t)&  = \mathbb{E}\bigg[ \exp 
 \Big( is \sum_{|n|\geq N} \jb{\wt{n}}^{-2} |g_n|^2
 + it \sum_{|n| \geq N} \jb{\wt{n}}^{-2} \wt{n} |g_n|^2\Big) \bigg] \notag\\
& = \prod_{|n| \geq N} \mathbb{E}
\Big[ e^{ i (s  \jb{\wt{n}}^{-2} 
 + t \jb{\wt{n}}^{-2} \wt{n} ) |g_n|^2} \Big] \notag\\
& = \prod_{n \geq N}
\frac{1}{\big(1-2i \jb{\wt{n}}^{-2} (s  + t \,  \wt{n} ) \big)
 \big(1-2i \jb{\wt{n}}^{-2} (s  
 - t \, \wt{n} )\big)}\label{W2}.
\end{align}

\noi
For any $n \geq N$, we have
$ \max (s+t \wt{n}, s-t \wt{n}) \geq \max (s, t \wt{n}).$
Also, note that each factor in \eqref{W2} is bounded by 1.
Thus, considering the terms for $n = N, \dots, N+3$ in \eqref{W2}, we have
\begin{align*}
|\ft{f}_N(s, t)|
\leq C(N) \jb{s}^{-2}\jb{t}^{-2}, 
\end{align*}

\noi
where $C(N)$ is at most a power of $N$.
Therefore,  we have
$\| \ft{f}_N\|_{L^1_{s, t}} <C'(N) <\infty.$
Note that $C'(N)$ is at most a power of $N$.
We use this fact in Subsection \ref{SUBSEC:2.2}.
\end{proof}

By Lemma \ref{LEM:Xft}, we have, for any $N\geq 0$, 
\begin{align}
\notag
& \frac{P\Big( 
\sum_{|n|\geq N} \jb{n}^{-2} |g_n|^2\in A_\eps(\wt{a}), 
\,   \sum_{|n|\geq  N} \jb{n}^{-2} n |g_n|^2  \in B_\eps(\wt{b})\Big)}{|A_\eps(\wt{a})\times B_\eps(\wt{b})|} \\
& \hphantom{XXXXXX} = 
\frac{1}{|A_\eps(\wt{a})\times B_\eps(\wt{b})|}
\int_{A_\eps(\wt{a})\times B_\eps(\wt{b})} f_{N}(a', b') da' db'
\too f_{N}(\wt{a}, \wt{b}), \label{W4}
\end{align}

\noi
as $\eps \to 0$.
By the uniform continuity of $f_N$, 
this convergence is uniform in $\wt{a}$ and $\wt{b}$.

In taking the limit of \eqref{Wiener1} as $\eps \to 0$,
the expression $f_0(a, b)$, i.e. \eqref{Xdensity} with $N = 0$,  appears in the denominator.
Hence, we need to show that $f_{0}(a, b) > 0$ for any $a>0$ and  $b \in \R$.
Indeed, we have
\begin{proposition}\label{LEM:POS}
Let $a> 0$ and $b \in \R$.
Then, we have $f_{0}(a, b) > 0$.
\end{proposition}

\noi
Proposition \ref{LEM:POS} is intuitively obvious.
However, since $f_0$ involves an infinite number of random variables,
we were not able to find any reference.
The proof will be given at the end of this subsection.

\medskip
Putting everything together, we have
\begin{align}
\frac{P\Big( 
\sum_{|n|\geq N+1} \jb{\wt{n}}^{-2} |g_n|^2\in A_\eps(\wt{a}), 
\,   \sum_{|n|\geq N+1} \jb{ \wt{n}}^{-2}\wt{n} |g_n|^2  \in B_\eps(\wt{b})\Big)}
{P\Big( \sum_n \jb{\wt{n}}^{-2} |g_n|^2\in A_\eps(a), 
\,   \sum_n \jb{\wt{n}}^{-2} \wt{n} |g_n|^2  \in B_\eps(b)\Big)}
\too \frac{f_{N+1}(\wt{a}, \wt{b})}{f_{0}(a, b)},
\label{W5} 
\end{align}

\noi
where the convergence is uniform in $\wt{a}$ and $\wt{b}$.
Moreover, 
the left hand side of \eqref{W5} is uniformly bounded
for small $\eps > 0$ (for {\it fixed} $a$ and $b$), 
since
$\|f_{N+1}\|_{L^\infty}  \leq \|\ft{f}_{N+1}\|_{L^1} <\infty$
and $f_0(a, b) > 0$.
Hence, by \eqref{Gauss4},  \eqref{Wiener1}, and Lebesgue dominated convergence theorem, 
we have
\begin{align}
P_0(E) = \lim_{\eps\to0}
P_\eps(E) & = \int_F 
\frac{f_{N+1}(\wt{a}, \wt{b})}{f_{0}(a, b)} \frac{e^{-\frac{1}{2}\sum_{|n|\leq N}  |\xi_n|^2}}
{(2\pi)^{2N+1}}\prod_{|n|\leq N} d \xi_n.
\notag 
\end{align}

\noi
This shows that $P_0$ is a well-defined probability measure.
Lastly, note that it basically follows from the definition that $P_\eps$ converges weakly to $P_0$.

\medskip

We will need the following lemma for the proof of  Proposition \ref{LEM:POS}. 
\begin{lemma} \label{LEM:ZERO}
Assume that $f(a^*, b^*) = 0$ for some $a^* > 0$ and $b^* \in \R$.
Then, there exists sufficiently large $N_0 \in \mathbb{N}$
such that $f_N(a, b) = 0$ on
\begin{equation} \label{B}
 B:= \{(a, b) \in \R_+\times \R: a \leq \tfrac{1}{2}a^*, \ |b| \leq |b^*|+1\}
\end{equation}

\noi
for all $N \geq N_0$.
\end{lemma}

\begin{proof}
First, note that, by symmetry, we have 
\begin{align} \label{Xsymmetry}
f_N(a, b) = f_N(a, -b)
\end{align}

\noi
for any $a, b \in \R$ and $N \geq 0$.
Defining $X_N$ and $Y_N$ by 
\begin{equation}\label{XY}
X_N = \sum_{|n|\geq N} \jb{\wt{n}}^{-2} |g_n|^2
\quad \text{and} \quad
Y_N = \sum_{|n|\geq N} \jb{\wt{n}}^{-2} \wt{n} |g_n|^2 ,
\end{equation}

\noi
we have
$X_0 = X_1 + |g_0|^2$ and $Y_0 = Y_1$.
Note that $X_1$ and $|g_0|^2$ are independent.
Thus, we can write $f_{0}$ as 
$f_{0} = f_1 *_a \chi_2^2$,
where $\chi_2^2$ is the density for the (rescaled) chi square 
distribution with two degrees of freedom,
corresponding to $|g_0|^2 = (\text{Re} \, g_0)^2 + (\text{Im} \, g_0)^2$,
and $*_a$ denotes the convolution only in the first variable of $f_1$. 
Recall that $\chi_2^2 (x)>0$ for $x > 0$ and $= 0$ for $x < 0$.

Now, suppose that $f_{0}(a^*, b^*) = 0$ for some $a^*>0$ and $b^*\in \R$.
By \eqref{Xsymmetry}, assume that $b^* \geq 0$.
Then, from 
\begin{align*}
0 = f_{0}(a^*, b^*) & = \int_{x>0} f_1(a^*-x, b^*) 
\chi_2^2(x) dx
\end{align*}

\noi
and the positivity of $\chi^2_2$ on $\R_+$,
we have 
$f_{1}(a, b^*) = 0$ for $a \leq  a^*$.
(Recall that $f_1$ is continuous by Lemma \ref{LEM:Xft}.)

Let $c_1(n)$ and $c_2(n)$ be given by
\begin{equation} \label{C}
c_1 (n) = (1+4\pi^2n^2)^{-1} \quad \text{and}  \quad c_2(n) =  2\pi n c_1(n), \quad n \in \mathbb{N}.
\end{equation}

\noi
Then, 
from \eqref{XY}, we have 
\[X_1 = X_2 + c_1(1)(|g_1|^2 + |g_{-1}|^2)
\quad \text{and}\quad
Y_1 = Y_2 + c_2(1) (|g_1|^2 - |g_{-1}|^2).\]

\noi
Since $f_1(a^*, b^*) = 0$, we have
\begin{align} \label{Xdensity2}
0 = f_1(a^*, b^*)
= \int_0^{\infty}\int_0^{\infty} f_2\big(a^* - c_1(1)(x + y), b^*- c_2(1)(x - y)\big) \chi^2_2(x)\chi^2_2(y)dx dy.
\end{align}

\noi
By change of variables $p = x+ y$ and $q = x- y$,
we can write  \eqref{Xdensity2}  as 
\begin{align*} 
0 = c\iint_{\substack{p>0\\|q|\leq p}} f_2(a^* - c_1(1) p, b^*- c_2(1) q) \chi^2_2(\tfrac{p+q}{2})\chi^2_2(\tfrac{p-q}{2})dp dq.
\end{align*}

\noi
This implies that $f_2(a, b) = 0$ 
on a triangular region
\begin{equation*}
A_2 : = \{ (a, b) \in \R_+ \times \R: a \leq a^*, \ |b - b^*| \leq 2\pi (a^* - a)\}.
\end{equation*}

\noi
In particular, $f_2(a^*, b^*) = 0$.
From \eqref{XY}, we have 
\[X_2 = X_3 + c_1(2)(|g_2|^2 + |g_{-2}|^2)
\quad \text{and}\quad
Y_2 = Y_3 + c_2(2) (|g_2|^2 - |g_{-2}|^2),\]

\noi
where $c_1(2)$ and $c_2(2)$ are as in \eqref{C}.
Since $f_2(a^*, b^*) = 0$, we have
\begin{align} \label{Xdensity4}
0 = f_2(a^*, b^*)
= \int_0^{\infty}\int_0^{\infty} f_3\big(a^* - c_1(2)(x + y), b^*- c_2(2)(x - y)\big) \chi^2_2(x)\chi^2_2(y)dx dy.
\end{align}

\noi
Once again, by change of variables $p = x+ y$ and $q = x- y$,
we can write  \eqref{Xdensity4}  as 
\begin{align*} 
0 = c\iint_{\substack{p>0\\|q|\leq p}} f_3(a^* - c_1(2) p, b^*- c_2(2) q) \chi^2_2(\tfrac{p+q}{2})\chi^2_2(\tfrac{p-q}{2})dp dq.
\end{align*}

\noi
This implies that $f_3(a, b) = 0$ 
on a triangular region
\begin{equation*}
A_3 : = \{ (a, b) \in \R_+ \times \R: a \leq a^*, \ |b - b^*| \leq 4\pi (a^* - a)\}.
\end{equation*}

\noi
In particular, we have $f_3(a^*, b^*) = 0$
and thus we can repeat the argument.
In general, from $f_N(a^*, b^*) = 0$,
we can show that $f_{N+1}(a, b) = 0$ 
on a triangular region
\begin{equation*}
A_{N+1} : = \{ (a, b) \in \R_+ \times \R: a \leq a^*, \ |b - b^*| \leq 2\pi N (a^* - a)\}
\end{equation*}

\noi
by simply noting $c_2(N)/c_1(N) = 2\pi N$.
By symmetry \eqref{Xsymmetry}, we have $f_{N+1}(a, b) = 0$
also on 
\begin{equation*}
\wt{A}_{N+1} : = \{ (a, b) \in \R_+ \times \R: a \leq a^*, \ |b + b^*| \leq 2\pi N (a^* - a)\}
\end{equation*}

\noi
Finally, by choosing $N_0$ large such that 
$\pi N_0 a^* \geq \max (1, b^*)$,
we see that 
$B \subset A_{N} \cup \wt{A}_{N}$ for $N \geq N_0$
and hence $f_N(a, b) = 0$ on $B$ for $N\geq N_0$.
\end{proof}

Finally, we conclude this subsection by presenting the proof of Proposition \ref{LEM:POS}.
\begin{proof}[Proof of Proposition \ref{LEM:POS}]
Suppose that $f_{0}(a^*, b^*) = 0$ for some $a^*>0$ and $b^*\in \R$.
By Lemma \ref{LEM:ZERO}, there exists $N_0 \in \mathbb{N}$
such that 
$f_N = 0$ on $B$ for all $N\geq N_0$, 
where $B$ is defined in \eqref{B}.
Recall that $f_N$ is nonnegative and $f_N(a, b) = 0$ for $a< 0$.
Then, by $(a, b) \in \R_+ \times \R \subset B \cup \{ a >\frac{1}{2}a^*\}
\cup \{ |b| \geq |b^*|+1 \}$, we have
\begin{align}
1 & = \int_\R \int_0^\infty f_N(a, b) da db\notag \\
& \leq \iint_{B} f_N(a, b) da db
+ \iint_{a > \frac{1}{2}a^*} f_N(a, b) da db
+ \iint_{|b| > |b^*|+1 } f_N(a, b) da db \notag \\
& = 0+P\big(X_N >  \tfrac{1}{2}a^*\big)
+ P\big( |Y_N| > |b^*|+1 \big), \label{Z}
\end{align}

\noi
for all $N \geq N_0$, where $X_N$ and $Y_N$ are as in \eqref{XY}.
Once we prove
\begin{align}
P\big(X_N >  \tfrac{1}{2}a^*\big) & <\tfrac{1}{2},  \label{X}\\
 P\big( |Y_N| > |b^*|+1 \big) & <\tfrac{1}{2}, \label{Y}
\end{align}

\noi
for some $N$, \eqref{Z} together with \eqref{X} and  \eqref{Y} leads to a contradiction, 
and hence $f_0(a, b) > 0$ for all $a >0$ and $b \in \R$.

Therefore, it remains to prove \eqref{X} and \eqref{Y} for large $N$.
First, we prove \eqref{Y}.
Write $Y_N$ as 
\[ Y_N = \sum_{n \geq N} \frac{2\pi n}{1+4\pi^2 n^2} \big(|g_n|^2 - |g_{-n}|^2\big).\]

\noi
Since $\mathbb{E} \big[ |g_n|^2 - |g_{-n}|^2\big] = 0$, we have
$\mathbb{E}\big[|Y_N|^2\big] \leq C N^{-1}$.
Then, by Chebyshev's inequality, we conclude that 
\[P\big( |Y_N| > |b^*|+1 \big) \leq \mathbb{E}\big[|Y_N|^2\big] \leq C N^{-1}.\]

\noi
Hence, there exists $N_1$ such that \eqref{Y} holds for all $N\geq N_1$.

Next, we prove \eqref{X}.
Fix large dyadic $N_2 = 2^k$ (to be chosen later).
Let $\s_j = C 2^{-\frac{1}{2} j}$ 
such that $\sum_{j = 1}^\infty \s_j = 1$.
Then, for $N \geq N_2$, we have
\begin{align*}
P\big(X_N >  \tfrac{1}{2}a^*\big)
& \leq \sum_{j = k}^\infty 
P \bigg( \Big(\sum_{2^j \leq |n| <2^{j+1}} (1+4\pi^2 n^2)^{-2} |g_n|^2 \Big)^\frac{1}{2}
> \tfrac{1}{2}\s_j a^*\bigg)\\
& \leq \sum_{j = k}^\infty 
P \bigg( \Big(\sum_{2^j \leq |n| <2^{j+1}}  |g_n|^2 \Big)^\frac{1}{2}
> c_{a^*} \s_j 2^j \bigg),
\end{align*}

\noi
where $c_{a^*}>0$ is a constant depending only on $a^*$.
By the large deviation estimate (e.g. see Lemma 4.2 in \cite{OQV}), we obtain
\begin{align*}
P\big(X_N >  \tfrac{1}{2}a^*\big)
& \leq \sum_{j = k}^\infty 
e^{-c'_{a^*} \s_j^2 2^{2j}}
 \leq 
e^{-\wt{c}_{a^*}  \, 2^{k}} <\tfrac{1}{2}
\end{align*}

\noi
for sufficiently large $k\in \mathbb{N}.$
By choosing $N\geq \max (N_0, N_1, N_2)$, 
\eqref{Z} together with \eqref{X} and  \eqref{Y} leads to a contradiction. 
This completes the proof of Proposition \ref{LEM:POS}.
\end{proof}

\subsection{Gibbs measure conditioned on mass and momentum}
\label{SUBSEC:2.2}

In the previous subsection, we constructed the Wiener measure $P_0$
conditioned on mass and momentum
as a limit of conditioned Wiener measures $P_\eps$.
In this subsection, we define
the conditioned Gibbs measure $\mu_0 = \mu_{a, b}$
by \eqref{Gibbs1}.
In the defocusing case, 
\eqref{Gibbs1} defines a probability measure. 
In the focusing case, 
however, 
we need to show \eqref{weight1};
the weight $e^{\frac{1}{p} \int_\T |u|^p }$ is
integrable with respect to $P_0$
for $p\leq 6$
(with sufficiently small mass when $p = 6$.)

Bourgain \cite{B2} proved a similar integrability result 
of the weight $e^{\frac{1}{p} \int_\T |u|^p }$
with respect to the (unconditioned) Wiener measure $P$ in \eqref{Gauss1}
via dyadic pigeonhole principle
and a large deviation estimate. 
In the following, we also use dyadic pigeonhole principle
and a large deviation estimate
(for the conditioned Wiener measure $P_0$)
to show that the conditioned Gibbs measure $\mu_0$
is a well-defined probability measure.
Indeed, 
Lemma \ref{LEM:devi} below establishes a uniform large deviation estimate for $P_\eps$, $\eps>0$,
and we prove the $L^1$-boundedness of the weight $e^{\frac{1}{p} \int_\T |u|^p }$
with respect to $P_\eps$, uniformly in sufficiently small $\eps>0$.
See \eqref{I1}.

First, we present a {\it uniform} large deviation lemma
for the conditioned Wiener measure $P_\eps$, $\eps > 0$.

\begin{lemma} \label{LEM:devi}
Let $R \geq 5 N^\frac{1}{2}$ and $M \sim N$.
Then, we have
\begin{equation} \label{devi1}
P_\eps\bigg( \sum_{|n-M| \leq N } |g_n|^2 \geq R^2\bigg) 
\leq C e^{-\frac{1}{8} R^2}
\end{equation}

\noi
uniformly for  sufficiently small $\eps\geq0$.
\end{lemma}

\begin{proof}
By Chebyshev's inequality, we have
\begin{equation} \label{devi0}
P_\eps\bigg( \sum_{|n-M| \leq N } |g_n|^2 \geq R^2\bigg) 
\leq e^{-t R^2} \mathbb{E}_{P_\eps} \Big[e^{t \sum_{|n-M| \leq N } |g_n|^2}\Big].
\end{equation}

\noi
Set $t = \frac{1}{4}$.
We estimate 
$\mathbb{E}_{P_\eps} \Big[e^{\frac{1}{4} \sum_{|n-M| \leq N } |g_n|^2}\Big]$
in the following.
As in \eqref{Wiener1}, we can write it as
\begin{align}
\mathbb{E}_{P_\eps} & \Big[e^{\frac{1}{4} \sum_{|n-M| \leq N } |g_n|^2}\Big]\notag \\
 & = \int_{\C^{2N+1}}
\frac{P\Big( 
\sum_{|n-M|\geq N+1 } \jb{\wt{n}}^{-2} |g_n|^2\in A_\eps(\wt{a}), 
\,   \sum_{|n-M|\geq N+1} \jb{ \wt{n}}^{-2}\wt{n} |g_n|^2  \in B_\eps(\wt{b})\Big)}
{P\Big( \sum_n \jb{\wt{n}}^{-2} |g_n|^2\in A_\eps(a), 
\,   \sum_n \jb{\wt{n}}^{-2} \wt{n} |g_n|^2  \in B_\eps(b)\Big)}
\notag \\
& \hphantom{XXXXXXX}
\times \frac{e^{-\frac{1}{4}\sum_{|n-M| \leq N }  |\xi_n|^2}}{(2\pi)^{2N+1}}
\prod_{|n-M|\leq N} d \xi_n,
\label{devi2}
\end{align}

\noi
where $\wt{a}$ and $\wt{b}$ are given by
\begin{equation}
\wt{a} = a - \sum_{|n-M| \leq N} \jb{\wt{n}}^{-2} |\xi_n|^2, \quad
\text{and} \quad
\wt{b} = b - \sum_{|n-M| \leq N} \jb{\wt{n}}^{-2} \wt{n} |\xi_n|^2.
\end{equation}

\noi
By repeating the argument in Subsection 2.1, 
we can show that the right hand side of \eqref{devi2}
is uniformly bounded for small $\eps > 0$.

More precisely, define the density $\wt{f}_N(a, b)$
by 
\[\wt{f}_N(a, b) \, da  db = P\bigg( 
\sum_{|n-M|\geq N} \jb{\wt{n}}^{-2} |g_n|^2\in da, 
\,   \sum_{|n-M|\geq N} \jb{\wt{n}}^{-2} \wt{n} |g_n|^2  \in db\bigg).\]

\noi
Then, as in Subsection 2.1, one can prove
\begin{align}
\frac{P\Big( 
\sum_{|n-M|\geq N+1} \jb{\wt{n}}^{-2} |g_n|^2\in A_\eps(\wt{a}), 
\,   \sum_{|n-M|\geq N+1} \jb{ \wt{n}}^{-2}\wt{n} |g_n|^2  \in B_\eps(\wt{b})\Big)}
{P\Big( \sum_n \jb{\wt{n}}^{-2} |g_n|^2\in A_\eps(a), 
\,   \sum_n \jb{\wt{n}}^{-2} \wt{n} |g_n|^2  \in B_\eps(b)\Big)}
\too \frac{\wt{f}_{N+1}(\wt{a}, \wt{b})}{f_{0}(a, b)},
\label{devi3}
\end{align}

\noi
where the convergence is uniform in $\wt{a}$ and $\wt{b}$.
Moreover, 
by showing $\|\wt{f}_N\|_{L^\infty} <\infty$ as before,
we see that the left hand side of \eqref{devi3} is uniformly bounded
for small $\eps > 0$. 
(Recall that $a$ and $b$ are fixed.)
By \eqref{devi2},  \eqref{devi3}, and Lebesgue dominated convergence theorem, 
we have
\begin{align*}
\lim_{\eps\to0}
\mathbb{E}_{P_\eps}  \Big[e^{\frac{1}{4} \sum_{|n-M| \leq K } |g_n|^2}\Big] 
  & = \int_{\C^{2N+1}} 
\frac{\wt{f}_{N+1}(\wt{a}, \wt{b})}{f_{0}(a, b)} \frac{e^{-\frac{1}{4}\sum_{|n-M|\leq N}  |\xi_n|^2}}
{(2\pi)^{2N+1}}\prod_{|n-M|\leq N} d \xi_n \\
& \leq \frac{\|\wt{f}_{N+1}\|_{L^\infty}}{f_{0}(a, b)}
\int_{\C^{2N+1}} \frac{e^{-\frac{1}{4}\sum_{|n-M|\leq N}  |\xi_n|^2}}{(2\pi)^{2N+1}}\prod_{|n-M|\leq N} d \xi_n \\
& \leq \frac{\|\wt{f}_{N+1}\|_{L^\infty}}{f_{0}(a, b)}
2^{2N+1},
\end{align*}

\noi
where the last inequality follows from change of variables.
Also, by an analogous argument to the proof of Lemma \ref{LEM:Xft}, 
we see that 
$\|\wt{f}_{N+1}\|_{L^\infty} \leq\|(\wt{f}_{N+1})^{\wedge}\|_{L^1}$
is bounded at most by a power of $N$.
Hence, we have
\begin{equation} \label{devi4}
\mathbb{E}_{P_\eps}  \Big[e^{\frac{1}{4} \sum_{|n-M| \leq K } |g_n|^2}\Big] 
\lesssim 2^{3N}
\end{equation}

\noi
for all sufficiently small $\eps>0$.
Therefore, \eqref{devi1} follows from \eqref{devi0} and \eqref{devi4}
as long as $R^2 \geq (24 \ln 2) N$.
\end{proof}

\medskip

In the following, we show the $L^1$-boundedness of  the weight $e^{\frac{1}{p}\int_\T|u|^p}$
with respect to $P_\eps$, 
uniformly for sufficiently small $\eps \geq 0$,
for $p\leq 6$
(with sufficiently small mass when $p = 6$.)
This, in particular, shows that $\mu_\eps$ in \eqref{Gibbs2}
is a well-defined probability measure.

Note that it suffices to prove that 
\begin{align}
\int_0^\infty  & e^{\lambda} \, P_\eps\bigg( \int_\T |u|^p\ge p \ld \bigg) d \lambda \notag \\
& = \int_0^\infty  e^{\lambda} \,P\bigg( \int_\T |u|^p\ge p \ld\, \Big| 
\int_\T |u|^2 \in A_\eps(a), 
\,   i \int_\T u \cj{u}_x  \in B_\eps(b)\bigg) d\ld \leq C_p < \infty
\label{I1}
\end{align}

\noi
for all sufficiently small $\eps>0$.
The estimate \eqref{I1} follows once we prove
\begin{equation}
P_\eps\bigg( \int_\T |u|^p\ge p\ld \bigg)
\leq \begin{cases}
C e^{-c \ld^{1+\dl}} & \text{when } p < 6.\\
Ce^{-(1+\dl) \ld} & \text{when } p = 6.
\end{cases}
\label{I0}
\end{equation}

\noi
for  $\ld> 1$ (with some $\dl > 0$),
uniformly in small $\eps > 0$.

Before proving \eqref{I0}, 
let us introduce some notations.
Given $M_0 \in \mathbb{N}$, 
let $\proj_{>M_0}$ denote the Dirichlet projection onto the frequencies $\{|n| > M_0\}$.
i.e.  $\proj_{>M_0} u = \sum_{|n|>M_0} \ft{u}_n e^{2\pi i n x}.$
$\proj_{\leq M_0}$ is defined in a similar manner. 
Given $j \in \mathbb{N}$, let $M_j = 2^j M_0$.  
We use the notation $|n|\sim M_j$ to denote the set of integers $|n|\in (M_{j-1}, M_j]$,
and denote by $\proj_{M_j}$ the Dirichlet projection onto the dyadic block $(M_{j-1}, M_j]$,
i.e.  
$\proj_{M_j} u = \sum_{|n| \sim M_j } \ft{u}_n e^{2\pi i n x}.$

\medskip

Without loss of generality, assume $\eps \leq a$.
Then, we have 
$\int |u|^2 \leq 2a =: K $.
By Sobolev inequality (or equivalently, by Hausdorff-Young inequality followed
by H\"older inequality on the Fourier side in this particular case,) 
\begin{equation} \label{SOB1}
 \|\proj_{\le M_0} u \|_{L^p(\mathbb T)}
\le cM_0^{\frac12 - \frac1{p}} \|\proj_{\le M_0} u \|_{L^2(\mathbb T)}.
\end{equation}

\noi
Hence,  we have 
\begin{equation} \label{I2}
 \int_\T |\proj_{\le M_0} u|^p \le  \tfrac{p}{2} \ld \quad \text{on }\int_{\T} u^2 \leq K,
 \end{equation}

\noi
by choosing 
\begin{equation} \label{I22}
M_0= c_0\lambda^{\frac2{p-2}} K^{-\frac{p}{p-2} }
\sim c_0\lambda^{\frac2{p-2}} a^{-\frac{p}{p-2} }
\end{equation}
for some $c_0>0$.
Let $\sigma_j= C2^{-\dl j} $, $j=1,2,\ldots$ 
for some small $\dl>0$ where $C=C(\dl)$ is chosen such that $\sum_{j=1}^\infty \sigma_j=1$.  
Then, we have
\begin{equation}\label{I3}
   P_\eps \bigg(  \int_\T  |\proj_{ >M_0 } u |^p  > \tfrac{p}{2}\ld\bigg) 
 \le   \sum_{j=0}^\infty 
 P_\eps \Big(\| \proj_{ M_j } u \|_{L^p(\mathbb T)} >\sigma_j \big(\tfrac{p}{2}\ld\big)^{\frac1p}\Big).
\end{equation}

\noi
By Sobolev inequality as in \eqref{SOB1}, we have
\begin{equation}
\| \proj_{M_j} u \|_{L^p(\mathbb T)} \le cM_j^{\frac12 - \frac1p}  \| \proj_{M_j} u \|_{L^2(\mathbb T)}. 
\label{I4}
\end{equation}

\noi
From \eqref{G2}, we have 
\begin{equation}
\label{I5}
\| \proj_{M_j} u \|^2_{L^2(\mathbb T)} =\sum_{|n|\sim M_j} |\hat{u}_n|^2
=\sum_{|n|\sim M_j}  \big(1+ (2\pi n)^2\big)^{-1} |g_n|^2 .
\end{equation}  

\noi
From  \eqref{I4} and \eqref{I5}, the right hand side of \eqref{I3} is bounded by 
\begin{equation}\label{I6}
\sum_{j=0}^\infty P_\eps \bigg(\sum_{|n|\sim M_j} |g_n|^2  \ge R_j^2 \bigg),\quad 
\text{where } R_j := c' \sigma_j \lambda^{\frac1p}M_j^{\frac1p - \frac12} 
(1+ M^2_j )^{1/2}.
\end{equation}

\noi
Note that $R_j \gtrsim M_j^{\frac{1}{2}+\frac{1}{p}}
\gg M_j^\frac{1}{2}$. By applying Lemma \ref{LEM:devi} to \eqref{I6}, we obtain
\begin{align}
P_\eps \bigg(  \int_\T  |\proj_{ >M_0 } u |^p  > \tfrac{p}{2} \ld \bigg) 
& \lesssim \sum_{j = 0}^\infty e^{-\frac{1}{8}R_j^2}
\lesssim \sum_{j = 0}^\infty e^{-c'' \s_j^2 \ld^\frac{2}{p}
M_j^{\frac{p+ 2}{p}}} \notag \\
& \lesssim \sum_{j = 0}^\infty e^{-\wt{c} (2^j)^{\frac{p+2}{p} -2\dl} 
\ld^{\frac{2}{p}} M_0^{\frac{p+2}{p}}}
 \lesssim  e^{-c\ld^{\frac{2}{p}} M_0^{\frac{p+2}{p}}} 
\label{I8}
\end{align}

\noi
Hence, from \eqref{I8} and \eqref{I22}, we have
\begin{equation} \label{I7}
P_\eps \bigg(  \int_\T  | u |^p  >p \ld\bigg) 
\leq C \exp\big\{ -c  \, \ld^{1+\frac{6-p}{p-2}} a^{-\frac{p+2}{p-2}}\big\}
\end{equation}

\noi
and \eqref{I0} follows. 
Note that when $p = 6$, we need to take $a$ sufficiently small
such that the coefficient of $\lambda$ in \eqref{I7} is less than $-1$.

\subsection{Weak convergence}
Finally, we prove weak convergence of $\mu_\eps$ defined in \eqref{Gibbs2} to $\mu_0$.
Let $f$ be a bounded continuous function on $H^{\frac{1}{2}-\g}(\T)$ for some small $\g > 0$.

\medskip
We first consider the defocusing case.
If a sequence of functions $u_n$ converges to $u$ in $H^{\frac{1}{2}-\g}(\T)$ with $\g < p^{-1}$, 
then we have $u_n \to u$ in $L^p(\T)$
by Sobolev inequality.
Thus, 
$e^{-\int_\T |u|^p}$ is bounded and continuous on $H^{\frac{1}{2}-\g}(\T)$.
Then, by weak convergence of $P_\eps $ to $P_0$,
we have
\[Z_\eps = \int e^{-\frac{1}{p}\int_\T |u|^p} dP_\eps
\too \int e^{-\frac{1}{p}\int_\T |u|^p} dP_0 = Z_0 \quad \text{as } \eps \to 0.\]

\noi
Since $f(u)  e^{-\int_\T |u|^p}$ is also bounded and continuous on $H^{\frac{1}{2}-\g}(\T)$,
we have
\begin{align*}
\int f d\mu_\eps
=  Z_\eps^{-1} \int f (u) e^{-\frac{1}{p}\int_\T |u|^p} dP_\eps
\too  Z_0^{-1} \int f (u) e^{-\frac{1}{p}\int_\T |u|^p} dP_0
= \int f d\mu_0 \quad \text{as }  \eps \to 0.
\end{align*}

\noi
This shows that $\mu_\eps$ converges weakly to $\mu_0$
in the defocusing case.

\medskip

Next, we consider the focusing case.
First, we prove 
\begin{equation} \label{Z1}
Z_\eps = \int e^{\frac{1}{p}\int_\T |u|^p} dP_\eps
\too \int e^{\frac{1}{p}\int_\T |u|^p} dP_0 = Z_0 \quad \text{as } \eps \to 0.
\end{equation}

\noi
Let $g(u) = e^{\frac{1}{p}\int_\T |u|^p}$.
By Chebyshev's inequality with the uniform integrability \eqref{I0}, we have
\begin{equation}
\int_{g >B} g(u) dP_\eps \leq C B^{-\dl}
\end{equation}

\noi
for all small $\eps \geq 0$.
Then, \eqref{Z1} follows once we note that
\begin{align*}
|Z_\eps - Z_0| \leq 
\bigg|\int_{g >B} g(u) dP_\eps\bigg|
+ \bigg|\int_{g \leq B} g(u) (dP_\eps - dP_0 ) \bigg|
+ \bigg|\int_{g >B} g(u) dP_0\bigg|,
\end{align*}

\noi
where the second term goes to 0 by the weak convergence of $P_\eps$ to $P_0$.

Let $f$ be a bounded continuous function $f$ on $H^{\frac{1}{2}-\g}(\T)$.
Then,
by writing
\begin{align*}
\int f d\mu_\eps - \int f d\mu_0
& =  Z_\eps^{-1} \int f (u) g(u) dP_\eps
-  Z_0^{-1} \int f (u) g(u) dP_0\\
& =  Z_0^{-1} \bigg( \int f (u) g(u) dP_\eps
-  \int f (u) g(u) dP_0\bigg)\\
& \hphantom{X}
+ ( Z_\eps^{-1} -Z_0^{-1})\int f (u) g(u) dP_\eps,
\end{align*}

\noi
it follows from \eqref{Z1} that the second term on the right hand side goes to zero. 
The first term goes to zero
by the uniform integrability \eqref{I0} with Chebyshev's inequality as before.
Hence, $\mu_\eps$ converges weakly to $\mu_0$.
This completes the proof of Theorem \ref{thm1}.

\section{Proof of Theorem \ref{thm2}: Invariance of the conditioned Gibbs measures}

In this section, we show that the conditioned Gibbs measure $\mu_0$
is invariant under the flow of NLS \eqref{NLS1}.
In fact, one can directly establish the invariance of the conditioned Gibbs measure $\mu_0$
by following the argument developed by Bourgain \cite{B2, B3}.
This argument is based on 
approximating the PDE flow by finite dimensional Hamiltonian systems
with invariant finite dimensional Gibbs measures.
For such an argument, one needs 
the following large deviation estimate (with $\eps = 0$.)
\begin{lemma} \label{LEM:LD}
Let $s < \frac{1}{2}$. Then, we have
\begin{equation} \label{LD}
P_\eps \Big( \|u\|_{H^s} > \Ld \Big)  \leq C_s e^{-c\Ld^2},
\end{equation}

\noi
uniformly in small $\eps \geq 0$.
\end{lemma}

\begin{proof}
This basically follows from the proof of \eqref{I7} in Subsection \ref{SUBSEC:2.2}.
Given $s < \frac{1}{2}$, choose $p > 2$ such that $s = \frac{1}{2}-\frac{1}{p}$.
Then, we have
\begin{equation} \label{LD1}
 \|\proj_{\le M_0} u \|_{H^s(\mathbb T)}
\le cM_0^{\frac12 - \frac1{p}} \|\proj_{\le M_0} u \|_{L^2(\mathbb T)}.
\end{equation}

\noi
(Compare this with \eqref{SOB1}.)
By repeating the computation in Subsection \ref{SUBSEC:2.2}
(with $\Ld = \ld^\frac{1}{p}$), we obtain
\begin{equation} \label{LD2}
P_\eps \Big( \|u\|_{H^s} > \Ld \Big) \leq C_s \exp\big\{ -c  \, \Ld^{p(1+\frac{6-p}{p-2})} a^{-\frac{p+2}{p-2}}\big\}.
\end{equation}

\noi
Then, \eqref{LD} follows since $p(1+\frac{6-p}{p-2}) >2$ for $p > 2$.
\end{proof}

\noi
Bourgain's argument \cite{B2, B3}
requires a combination of PDE and probabilistic techniques.
In the following, however, we simply show how the invariance of the conditioned Gibbs measure $\mu_0$
follows, as a corollary, from a priori invariance of Gibbs measures $\mu_\eps$, $\eps > 0$.

\medskip 
\noi
$\bullet$ {\bf Case 1:} $p\leq 6$.
\quad 
In this case, the flow of \eqref{NLS1}
is globally defined in $H^{\frac{1}{2}-\dl}(\T)$ for small $\dl = \dl(p) > 0$,
thanks to \cite{B1, B5}.
Let $\mathcal{S}_t$ be the flow map of \eqref{NLS1}: $u_0 \mapsto u(t) = \mathcal{S}_{t} u_0$.
Then, $\mathcal{S}_t$ is well-defined and continuous
on $H^{\frac{1}{2}-\dl}(\T)$

Given a bounded continuous function $\phi$ on $H^{\frac{1}{2}-\dl}(\T)$, 
$\phi\circ \mathcal S_t$ is bounded and continuous on $H^{\frac{1}{2}-\dl}(\T)$.
By weak convergence of $\mu_\eps$ to $\mu_0$
and invariance of $\mu_\eps$ under the flow of \eqref{NLS1}, we have
\begin{equation*} 
\int \phi \, d\mu_0 = 
\lim_{\eps\to0}\int \phi \, d\mu_\eps = \lim_{\eps\to0} \int \phi\circ \mathcal S_t \, d\mu_\eps 
 = \int \phi\circ \mathcal S_t \, d\mu_0.
\end{equation*}

\noi
This proves invariance of $\mu_0$ for $p \geq 6$.

\medskip
\noi
$\bullet$ {\bf Case 2:} $p> 6$.
(This is relevant only in the defocusing case.)
\quad 

In this case, 
there is no a priori global-in-time flow of \eqref{NLS1}
on $H^{\frac{1}{2}-\dl}(\T)$.
However, by Bourgain's argument \cite{B2, B3}, 
$\mu_\eps$ is invariant under the flow of NLS \eqref{NLS1}
for each $\eps > 0$,
and we show invariance of $\mu_0$ as a corollary to the invariance of $\mu_\eps$, $\eps > 0$.

Let $K$ be a compact set in $H^{s}(\T)$ with $s = \frac{1}{2}-$.
Then, there exists $\Ld = \Ld(K) >0$ such that  $\|u\|_{H^s} \leq \Ld$ for $u \in K$.
By the (deterministic) local well-posedness \cite{B2}, 
there exists $t_0>0$ such that NLS \eqref{NLS1} is well-posed on $[0, t_0]$
for initial data $u_0$ with $\|u_0\|_{H^s} \leq \Ld + 1$.
Moreover, for each small $\theta > 0$, 
 there exists $\dl > 0$ such that 
\begin{equation} \label{ZZZ}
\mathcal{S}_{t_0} (K+B_\dl)\subset \mathcal{S}_{t_0} K+B_{\theta}.
\end{equation}

\noi
Then, by weak convergence of $\mu_\eps$ to $\mu_0$, we have
\begin{align*}
\mu_0(K) & \leq \mu_0(K+B_\dl)
\leq \liminf_{\eps \to 0} \mu_\eps(K+B_\dl)\\
\intertext{By invariance of $\mu_\eps$
and \eqref{ZZZ}, }
& = \liminf_{\eps \to 0} \mu_\eps\big(\mathcal{S}_{t_0} (K+B_\dl)\big)
\leq \liminf_{\eps \to 0} \mu_\eps(\mathcal{S}_{t_0} K+B_{\theta}\,)\\
& \leq \limsup_{\eps \to 0} \mu_\eps(\mathcal{S}_{t_0} K+B_{\theta}\,)
 \leq \limsup_{\eps \to 0} \mu_\eps(\mathcal{S}_{t_0} K+\cj{B_{\theta}}\,)\\
& \leq \mu_0(\mathcal{S}_{t_0} K+\cj{B_{\theta}}\,),
\end{align*}

\noi
where the last inequality follows once again from 
the weak convergence of $\mu_\eps$ to $\mu_0$.
By letting $\theta \to 0$, we have 
$\mu_0(K) \leq \mu_0(\mathcal{S}_{t_0} K)$.
Given arbitrary $t>0$, we can iterate the above argument and obtain 
$\mu_0(K) \leq \mu_0(\mathcal{S}_{t} K)$.
By the time-reversibility of the NLS flow, we obtain 
\[\mu_0(K) = \mu_0(\mathcal{S}_t K).\]

\noi
This proves invariance of $\mu_0$ for $p > 6$.

\end{document}